\documentclass[a4paper]{gtart}

\usepackage[inner=20mm, outer=20mm, textheight=245mm]{geometry}

\usepackage{amsmath}
\usepackage{amssymb}
\usepackage{amsthm,hyperref}

\newtheorem{Theorem}{Theorem}[section]
\newtheorem{Definition}[Theorem]{Definition}
\newtheorem{Conjecture}[Theorem]{Conjecture}
\newtheorem{Corollary}[Theorem]{Corollary}
\newtheorem{Lemma}[Theorem]{Lemma}
\newtheorem{Proposition}[Theorem]{Proposition}
\theoremstyle{remark}
\newtheorem{Remark}[Theorem]{Remark}

\newtheorem{Question}[Theorem]{Question}

\theoremstyle{remark}
\newtheorem*{claim}{Claim}

\newcommand{\N}{\mathbb{N}}
\newcommand{\R}{\mathbb{R}}
\newcommand{\Z}{\mathbb{Z}}

\newcommand{\C}{\mathbb{C}}
\newcommand{\E}{\mathbb{E}}
\newcommand{\Hyp}{\mathbb{H}}
\newcommand{\Sph}{\mathbb{S}}

\DeclareMathOperator{\PGL}{PGL}
\DeclareMathOperator{\PSL}{PSL}
\DeclareMathOperator{\GL}{GL}
\DeclareMathOperator{\SL}{SL}

\usepackage{listings}
\usepackage{color}

\lstdefinestyle{base}{
  emptylines=1,
  moredelim=**[is][\color{blue}]{@}{@},
  moredelim=**[is][\color{green}]{@@}{@@},
}

\usepackage{tikz}
\usepackage{tikz-cd}
\usetikzlibrary{decorations.markings}
\usetikzlibrary{arrows,shapes,positioning}
\usetikzlibrary{matrix}
\usetikzlibrary{cd}
\usetikzlibrary{arrows.meta}

\usepackage{bbm}
\usepackage{mathrsfs}

\newcommand{\smfrac}[2]{\mbox{\footnotesize$\displaystyle\frac{#1}{#2}$}}

\newcommand{\tmprod}[2]{\mbox{$\textstyle \prod\limits_{#1}^{#2}$}}
\newcommand{\mraisebox}[2]{\mbox{\raisebox{#1}{$#2$}}}

\begin{document}

\title{Linear representations of 3--manifold groups over rings}
\author{Stefan Friedl, Montek Gill and Stephan Tillmann}

\begin{abstract}
The fundamental groups of compact 3--manifolds are known to be residually finite. 
Feng Luo conjectured that a stronger statement is true, by only allowing finite groups of the form $\PGL_2(R),$ where $R$ is some finite commutative  ring with identity. We give an equivalent formulation of Luo's conjecture via faithful representations and provide various examples and a counterexample.
\end{abstract}

\makeshorttitle

\section{Introduction}
In this paper, by 3--manifold, we mean connected 3--manifold. A hyperbolic structure of finite volume on an orientable 3--manifold $M$ gives rise to a developing map $\text{dev}\co \widetilde M \to \Hyp^3$, where $\widetilde M$ is the universal cover of $M$, and an embedding $\text{hol}\co \pi_1(M) \to \operatorname{Isom}^+(\Bbb{H}^3)=\PSL_2(\C)$, which is the holonomy representation associated to the geometric structure and chosen developing map. For details, see, for example, \cite[Chapter 3]{Thurston} or \cite[Chapter 8]{rathypst}. In the case that $M$ is triangulated, the decomposition of $M$ into simplices lifts to one of $\widetilde M$ which gives rise to a labelling of the 0-skeleton of the lifted triangulation by elements of $\partial \overline{\Hyp}^3 = \C P^1$ and this labelling encodes all the information necessary to construct the holonomy representation; see \cite{ThurstonNotes} for the case of torus cusps,  and \cite{LuoTillmannYang} for the closed case. In \cite{Feng}, Luo generalises these labellings to labellings over $\mathbb{P}^1(R)$, the projective line over an arbitrary commutative ring with identity, $R$, and constructs representations into $\PGL_2(R)$. An example illustrating the strength of Luo's generalisation is given in \S\ref{sec:quaternionic}. Luo makes the following conjecture: 

\begin{Conjecture}[Luo~\cite{Feng}]\label{conj:luo} If $M$ is a compact $3$--manifold and $\gamma \in \pi_1(M) \setminus \{1\}$, there exists a finite commutative ring $R$ with identity and a homomorphism $\pi_1(M) \rightarrow \PGL_2(R)$ whose kernel does not contain $\gamma$.
\end{Conjecture}

At this point it is helpful to introduce the following definition.

\begin{Definition}\label{def:residual_props}
Given a group $G$, say that $G$ is \textit{residually $\PGL_2$--finite} if  for any $g \in G \setminus \{1\}$, there exists a finite commutative ring $R$ with identity and a homomorphism $G \rightarrow \PGL(2,R)$ whose kernel does not contain $g$. Define, in an analogous manner, \textit{residually $\PSL_2$--finite}, \textit{residually $\SL_2$--finite} and \textit{residually $\GL_2$--finite}.
\end{Definition}

Luo's Conjecture thus says that fundamental groups of compact 3--manifolds are residually $\PGL_2$--finite. In Section 2, we show that for most groups the different notions of residual finiteness are equivalent. More precisely, we show that for a finitely generated group $G$, we have the following implications: 
\begin{center}
 \begin{tikzpicture}
 \draw[-implies, double equal sign distance] (0.2,2.1) node[anchor=east] {\text{residually $\SL_2$--finite\,\,}} -- (3,2.1)  node[anchor=west] 
 {\text{\,\,residually $\GL_2$--finite}}
 node[anchor=south,midway] {\textit{}};
 \draw[-implies, double equal sign distance] (3,1.9) -- (0.2,1.9) node[anchor=north,midway] { $Z(G)=1$};
 \draw[implies-, double equal sign distance] (-1.25,1.7) -- (-1.25,0.75) node[anchor=north] {\hspace{-8mm} \text{residually $\PSL_2$--finite\,\,}} node[anchor=west,midway] {\text{}};
 \draw[-implies, double equal sign distance] (-1.5,1.7) -- (-1.5,0.75) node[anchor=east,midway] {\text{ $Z(G)$ 2-t.f.}};
 \draw[implies-, double equal sign distance] (4.5,1.7) -- (4.5,0.75) node[anchor = north] 
 {\hspace{13mm} \text{\hspace{-5mm} \,\,residually $\PGL_2$--finite}} node[anchor=east,midway] {\text{}};
 \draw[-implies, double equal sign distance] (4.75,1.7) -- (4.75,0.75) node[anchor=west,midway] {\text{ $Z(G) = 1$}};
 \draw[-implies,double equal sign distance] (0.4,0.6) -- (3,0.6) node[anchor=south,midway] {\text{}};
 \draw[-implies,double equal sign distance] (3,0.4) -- (0.4,0.4) node[anchor=north,midway] {\text{}};
 \end{tikzpicture}
\end{center}
where 2-t.f.\ means 2-torsion-free and $Z(G)$ denotes the centre of $G$. 

In the following let $K$ be one of the symbols $\SL_2$, $\GL_2$, $\PSL_2$, $\PGL_2$.
In this paper we investigate different types of groups and check whether they are residually $K$--finite. If a group is residually $K$--finite, then one can usually show so by writing down a representation. On the other hand the proof of Proposition~\ref{universal_rep} gives a practical approach to showing that a group is not residually $K$--finite. 

As applications we consider several classes of groups in \S\ref{sec:definitions}, giving both positive and negative results.  
For example in Theorem~\ref{thm:Sn} we show that the symmetric group $S_n$ in $n$ letters is residually $\PGL_2$--finite if and only if $n<5$.  We also show that some 3--manifold groups are residually $K$--finite for all choices of $K$. 
Our main result though is that  Conjecture~\ref{conj:luo} does not hold.
More precisely, we prove the following Theorem in \S\ref{sec:counterexample}. 

\begin{Theorem}
There exists a closed graph manifold $M$ such that $\pi_1(M)$ is not residually $K$--finite for $K=\SL_2$, $\GL_2$, $\PSL_2$, $\PGL_2$.
\end{Theorem}

We conclude this introduction with a short discussion of the linearity of fundamental groups of 3--manifolds. The fact that our counterexample is  a closed graph manifold is perhaps not surprising since it is still unknown whether fundamental groups of closed graph manifolds are linear. This raises the following question.

\begin{Question}
Does there exist a  counterexample to Conjecture~\ref{conj:luo} that is a prime 3--manifold but that is not a graph manifold?
\end{Question}

We also recall that  Thurston~\cite[Problem 3.33]{Kirbylist} asked
whether every finitely generated 3--manifold group has a faithful representation in
$\GL(4,\R)$. Button~\cite{Button} recently answered this question in the negative. More precisely, he showed that there exists a closed graph manifold $M$ such that $\pi_1(M)$ does not admit a faithful representation into $\GL(4,\R)$. 

We conclude this introduction with the following questions.

\begin{Question}
\begin{enumerate}
\item Does there exist a natural number $n$ such that the fundamental group of every compact 3--manifold  admits a faithful representation into $\GL(n,\C)$?
\item Does there exist a natural number $n$ such that the fundamental group of every compact 3--manifold is residually $\GL_n$--finite?
\end{enumerate}
\end{Question}

\subsection*{Acknowledgments.} 
The first author is supported by the SFB 1085 ``Higher invariants'' at the University of Regensburg, funded by the Deutsche Forschungsgemeinschaft (DFG).
The first author also wishes to thank Matthias Aschenbrenner, Philip Matura and Brendan Owens for helpful conversations. 
 The third author is partially supported under the Australian Research Council's Discovery funding scheme (project number DP140100158).
 

\section{Alternative characterisations}
\label{sec:characterisations}

We have the following result of Baumslag, proven in \cite{Baumslag} (Theorem 5.3, p.64).

\begin{Theorem}\label{Baumslag}
\textit{If a commutative ring $R$ is finitely generated as a $\mathbb{Z}$-algebra, then $R$ is residually finite; that is, for each non-zero $r \in R$, there exists a finite commutative $R^\prime$ and a ring map $\varphi\colon R \to R^\prime$ such that $\varphi(r) \neq 0$.}
\end{Theorem}

Let $K$ be one of the symbols $\SL_2$, $\GL_2$, $\PSL_2$, $\PGL_2$.

\begin{Proposition}\label{prop:reducing_the_ring}
Given a finitely generated group $G$, an element $g \in G \setminus \{1\}$ and a homomorphism $\rho\colon G \to \text{K}(R)$ for a commutative but not necessarily finite  ring $R$ such that $\rho(g) \neq 1$, there exists a homomorphism $\rho^\prime\colon G \to \text{K}(R^\prime)$ for a finite commutative $R^\prime$ such that $\rho^\prime(g) \neq 1$. In particular, if there exists a faithful representation $G \to K(R)$ for some commutative but not necessarily finite $R$, then $G$ is residually $K$--finite.
\end{Proposition}

\begin{proof}
Suppose first that $K$ is one of $\SL_2$, $\GL_2$. Let $\{ g_1,\dots,g_k \}$ be a set of generators for $G$. For $i=1,\dots,k$, let $A_i = \rho(g_i)$ and let $\widetilde{R}$ be the ring generated by the entries of $A_1,\dots,A_k$ as well as the elements $(\det A_1)^{-1},\dots,(\det A_k)^{-1}$; in the case that $K = \SL_2$, the inclusion of the determinants is superfluous. Note that $\widetilde{R}$ contains the entries of $A_1^{-1},\cdots,A_k^{-1}$. As such, we can restrict $\rho$ to attain a faithful representation $\rho^\prime\colon G \to \SL_2(\widetilde{R})$ and because $\widetilde{R}$ is a finitely generated $\mathbb{Z}$-algebra, it is residually finite by Theorem~\ref{Baumslag}.
\begin{enumerate}
\item[(i)] If $\rho^\prime(g) = \rho(g)$ has a non-zero off-diagonal entry, say $a$, we let $\phi\colon \widetilde{R} \to R^\prime$ be such that $\phi(a) \neq 0$ and $R^\prime$ is finite, then the image of $g$ under the map $G \xrightarrow{\rho^\prime} K(\widetilde{R}) \xrightarrow{\phi_*} K(R^\prime)$ is non-trivial.
\end{enumerate}
 Suppose then that $\rho^\prime(g)$ is diagonal, say $\text{diag}(a,b)$.
\begin{enumerate}
\item[(ii)] If $a-b \neq 0$, then  we can choose $\phi\colon \widetilde{R} \to R^\prime$ be such that $\phi(a-b) \neq 0$ and $R^\prime$ is finite; then the image of $g$ under the map $G \xrightarrow{\rho^\prime} K(\widetilde{R}) \xrightarrow{\phi_*} K(R^\prime)$ is non-trivial. 
\item[(iii)]  If $a=b$, say with both equal to $c \neq 1$, let $\phi\colon \widetilde{R} \to R^\prime$ be such that $\phi(c-1) \neq 0$ and $R^\prime$ is finite, then the image of $g$ under the map $G \xrightarrow{\rho^\prime} K(\widetilde{R}) \xrightarrow{\phi_*} K(R^\prime)$ is non-trivial. 
\end{enumerate}

Now suppose that $K$ is one of $\PSL_2$, $\PGL_2$.  Let $A_1,\dots,A_k$ be matrices that are representatives of $\rho(g_1),\dots,\rho(g_k)$ respectively and let $\widetilde{R}$ be the ring generated by the entries of $A_1,\dots,A_k$; note that $\widetilde{R}$ contains the entries of representatives for $\rho(g_1^{-1}),\cdots,\rho(g_k^{-1})$. Let $\iota\colon K(\widetilde{R}) \to K(R)$ denote the obvious embedding which leads to an isomorphism $K(\widetilde{R}) \cong \text{im}(\iota)$ and note that $\text{im}(\rho) \subseteq \text{im}(\iota)$; thus by restriction and composition we get a representation $\rho^\prime\colon G \to K(\widetilde{R})$ such that $\rho^\prime(g) \neq 1$.  The remainder of the argument is now precisely as the cases (i) and (ii) above. 
\end{proof}

Thus, if we define ``\textit{residually $K$}'' to mean the same thing as residual $K$--finiteness but with the finiteness requirement on the ring dropped, we have:

\begin{Corollary}\label{cor:res_K_res_K_finite}
A finitely generated group $G$ is residually $K$--finite if and only if it is residually $K$.
\end{Corollary}

We also have:

\begin{Proposition}\label{prop:res_finite_gives_faithful}
Let $G$ be a group.
If $G$ is a residually $K$--finite, then it admits a faithful representation $G \to K(R)$ for some, not necessarily finite, $R$.
\end{Proposition}

\begin{proof}
We start out with the following two observations:
\begin{enumerate}
\item $K$ is functorial in the ring, i.e.\ a ring homomorphism $\varphi\colon S\to S'$ induces a group homomorphism $\varphi_*\colon K(S)\to K(S')$.
\item If $S_i, i\in I$ is a family of rings then 
\[ \begin{array}{rcl} \tmprod{i\in I}{}  K(S_i)&\to & K\Big(\mraisebox{0.05cm}{\tmprod{i\in I}{}}S_i\Big)\\
(A_i)_{i\in I}&\mapsto & \tmprod{i\in I}{} A_i\end{array}\]
where the product matrix $\Pi_{i \in I}A_i$ is formed entrywise in the direct product of the  rings $S_i$, is well-defined and it is a group isomorphism.
\end{enumerate}
Now we turn to the actual proof of the proposition. For any $g \neq 1$ in $G$, we have a representation $\rho_g\colon G \to K(R_g)$ where $R_g$ is finite and $\rho_g(g) \neq 1$. Let $R = \prod_{g \neq 1}R_g$. We compose $(\rho_g)_{g \neq 1} \colon G \to \prod_{g \neq 1} K(R_g)$ with the group homomorphism given in (2) and we obtain representation $\rho\colon G\to K(R)$. 

We claim that  $\rho$ is faithful. Let $g\in G$ be non-trivial. Then the image of $\rho(g)$ under the projection map $K(R)\to K(R_g)$ equals $\rho_g(g)$, hence it is non-trivial.
\end{proof}

Combining Propositions~\ref{prop:reducing_the_ring} and \ref{prop:res_finite_gives_faithful}, we have:

\begin{Corollary}\label{iff_faithful_rep}
Suppose $G$ is a finitely generated group. Then $G$ is residually $K$--finite if and only if it admits a faithful representation $G \to K(R)$ for some, not necessarily finite, $R$.
\end{Corollary}

\begin{Proposition}\label{universal_rep}
Suppose $G$ is a finitely generated group.
Then there exists a commutative ring $S_K$, an ideal $I_K \trianglelefteq S_K$ and a map $\varphi_K\colon G \rightarrow K(S_K/I_K)$ such that any representation $G \to K(R)$ factors through $\varphi_K$; that is, for each $\rho\colon G \to K(R)$, there exists a mediating map $\psi\colon K(S_K/I_K) \to K(R)$ such that the following diagram commutes:

\begin{center}
 \begin{tikzpicture}
 \draw[->] (0.2,2) node[anchor=east] {$G$} -- (2,2)  node[anchor=west] 
 {$K(S_K/I_K)$}
 node[anchor=south,midway] {$\varphi_K$};
 \draw[->] (0.1,1.7) -- (2,0.75) node[anchor=east,midway,yshift=-3mm,xshift=1mm] {$\rho$};
 \draw[->] (2.7,1.7) -- (2.7,0.75) node[anchor = north] 
 {$K(R).$} node[anchor=west,midway] {$\psi$};
 \end{tikzpicture}
\end{center}
\end{Proposition}

\begin{proof}
Let $G = \langle g_1,\dots,g_n \: | \: r_i=e, i\in I \rangle$ and suppose first that $K = \SL_2$. Let \[S_{\SL_2} = \mathbb{Z}[x_{1a},x_{1b},x_{1c},x_{1d},\dots,x_{na},x_{nb},x_{nc},x_{nd}]\] and then define \[p(g_i) = \left(\begin{array}{cc} x_{ia} & x_{ib} \\ x_{ic} & x_{id} \end{array}\right),\quad p(g_i^{-1}) = \left(\begin{array}{cc} \hfill x_{id} & -x_{ib} \\ -x_{ic} & \hfill x_{ia} \end{array}\right).\] Define also $p(r_i)$  by setting that $p$ be multiplicative and then set 
\[I_{\SL_2} = \left\langle \{\det\,p(g_i)-1\}_i \cup \{(p(r_i)-1)_{k,l}\}_{i,k,l}\right\rangle.\]
Then we set \[\varphi_{\SL_2}\colon G \to \SL_2(S_{\SL_2}/I_{\SL_2})\colon g_i \mapsto \left(\begin{array}{cc} \overline{x_{ia}} & \overline{x_{ib}} \\ \overline{x_{ic}} & \overline{x_{id}} \end{array}\right)\]
which can be checked to be well-defined. Now, suppose that $\rho\colon G \to \SL_2(R)$ is given. Let 
$\rho(g_i) = (a^i_{kl})_{kl}$ and define $q\colon S_{\SL_2}/I_{\SL_2} \to R$ by
\[
1,x_{1a},x_{1b},x_{1c},x_{1d},\dots,x_{na},x_{nb},x_{nc},x_{nd} \mapsto \newline 1,a^1_{11},
a^1_{12},a^1_{21},a^1_{22},\dots, a^n_{11},a^n_{12},a^n_{21},a^n_{22}.
\]
The map $q$ is well-defined because $a^i_{11}a^i_{22} - a^i_{12}a^i_{21} - 1 = 0$ for each $i$ and because computation of $\rho(r_j)$ and $\rho(s_j)$ will give the required remaining equations defining $I$. This map $q$ induces a map \[\psi\colon \SL_2(S_{\SL_2}/I_{\SL_2}) \xrightarrow{q_{*}} \SL_2(R)\] by applying $q$ to each entry and one can then verify that $\psi \circ \varphi_{\SL_2} = \rho$ holds. \\

If $K = \GL_2$, we alter the definitions as follows: \[S_{\GL_2} = \mathbb{Z}[x_{1a},x_{1b},x_{1c},x_{1d},\dots,x_{na},x_{nb},x_{nc},x_{nd},y_1,\dots,y_n],\] \[p(g_i) = \left(\begin{array}{cc} x_{ia} & x_{ib} \\ x_{ic} & x_{id} \end{array}\right),\quad p(g_i^{-1}) = y_i\left(\begin{array}{cc}\hfill x_{id} & -x_{ib} \\ -x_{ic} & \hfill x_{ia} \end{array}\right),\] $p(r_i)$  are defined by setting that $p$ be multiplicative, \[I_{\GL_2} = \left\langle \{(\det\,p(g_i))y_i-1\}_i \cup \{(p(r_i)-1)_{k,l}\}_{i,k,l}\right\rangle,\] \[\varphi_{\GL_2}\colon G \to \GL_2(S_{\GL_2}/I_{\GL_2})\colon g_i \mapsto \left(\begin{array}{cc} \overline{x_{ia}} & \overline{x_{ib}} \\ \overline{x_{ic}} & \overline{x_{id}} \end{array}\right)\] and finally given $\rho\colon G \to \GL_2(R)$ and $\rho(g_i) = (a^i_{kl})_{kl}$, $q\colon S_{\GL_2}/I_{\GL_2} \to R\colon 1,x_{ia},x_{ib},x_{ic},x_{id},y_i \mapsto 1,a^i_{11},a^i_{12},a^i_{21},a^i_{22},(a^i_{11}a^i_{22}-a^i_{12}a^i_{21})^{-1}$ and $\psi = q_{*}$. \\

If $K = \PSL_2$, we alter the definitions as follows: \[S_{\PSL_2} = \mathbb{Z}[x_{1a},x_{1b},x_{1c},x_{1d},\dots,x_{na},x_{nb},x_{nc},x_{nd},\{\lambda_i\}_{i\in I}],\] \[p(g_i) = \left(\begin{array}{cc} x_{ia} & x_{ib} \\ x_{ic} & x_{id} \end{array}\right),\quad  p(g_i^{-1}) = \left(\begin{array}{cc}\hfill  x_{id} & -x_{ib} \\ -x_{ic} & \hfill x_{ia} \end{array}\right),\] $p(r_i)$ are defined by setting that $p$ be multiplicative, \[I_{\PSL_2} = \left\langle \{\det\,p(g_i)-1\}_i \cup \{\lambda_i^2-1\}_i \cup \{(p(r_i)-\lambda_i )_{k,l}\}_{i,k,l}\right\rangle,\] \[\varphi_{\PSL_2}\colon G \to \PSL_2(S_{\PSL_2}/I_{\PSL_2})\colon g_i \mapsto \left[\begin{array}{cc} \overline{x_{ia}} & \overline{x_{ib}} \\ \overline{x_{ic}} & \overline{x_{id}} \end{array}\right]\] and finally given $\rho\colon G \to \PSL_2(R)$, $\rho(g_i) = [a^i_{kl}]_{kl}$ and that the corresponding representative for $p(r_i)$ is equal to $\mu_i$ times the identity matrix, $q\colon S_{\PSL_2}/I_{\PSL_2} \to R\colon 1,x_{ia},x_{ib},x_{ic},x_{id},y_i \mapsto 1,a^i_{11},
a^i_{12},a^i_{21},a^i_{22},\mu_i$ and $\psi = q_{*}$. \\

If $K = \PGL_2$, we alter the definitions as follows: \[S_{\PGL_2} = \mathbb{Z}[x_{1a},x_{1b},x_{1c},x_{1d},\dots,x_{na},x_{nb},x_{nc},x_{nd},y_1,\dots,y_n,\{\lambda_i\}_{i\in I}],\] \[p(g_i) = \left(\begin{array}{cc} x_{ia} & x_{ib} \\ x_{ic} & x_{id} \end{array}\right),\quad p(g_i^{-1}) = \left(\begin{array}{cc} \hfill x_{id} & -x_{ib} \\ \hfill -x_{ic} & x_{ia} \end{array}\right),\] $p(r_i)$  are defined by setting that $p$ be multiplicative, \[I_{\PGL_2} = \left\langle \{(\det\,p(g_i))y_i-1\}_i \cup \{(p(r_i)-\lambda_i )_{k,l}\}_{i,k,l}\right\rangle,\] \[\varphi_{\PGL_2}\colon G \to \PGL_2(S_{\PGL_2}/I_{\PGL_2})\colon g_i \mapsto \left[\begin{array}{cc} \overline{x_{ia}} & \overline{x_{ib}} \\ \overline{x_{ic}} & \overline{x_{id}} \end{array}\right]\] and finally given $\rho\colon G \to \PGL_2(R)$, $\rho(g_i) = [a^i_{kl}]_{kl}$ and that the corresponding representative for $p(r_i)$ is equal to $\mu_i$ times the identity matrix, $q\colon S_{\PGL_2}/I_{\PGL_2} \to R\colon 1,x_{ia},x_{ib},x_{ic},x_{id},y_i,\lambda_i \mapsto 1,a^i_{11},
a^i_{12},a^i_{21},a^i_{22},(a^i_{11}a^i_{22}-a^i_{12}a^i_{21})^{-1},\mu_i$ and $\psi = q_{*}$.
\end{proof}

\begin{Remark}
We could use any other characteristic zero ring instead of $\Z$ for the coefficients in $S_K$. We will sometimes use $\C$ instead.
\end{Remark}

\begin{Proposition}\label{prop:iff_injection}
Suppose $G$ is a finitely generated group. Then $G$ is residually $K$--finite if and only if the map $\varphi_K\colon G \to K(S_K/I_K)$ above is an injection.
\end{Proposition}

\begin{proof}
By Proposition~\ref{iff_faithful_rep}, if $G$ it is residually $K$--finite, there exists a faithful $\rho\colon G \to K(R)$ for some $R$ so that, as $\rho$ factors through $\varphi_K$, $\varphi_K$ too is an injection. Conversely, if $\varphi_K$ is faithful, we apply Proposition~\ref{iff_faithful_rep} again with $\varphi_K$ as the injection to conclude that $G$ is residually $K$--finite.
\end{proof}

\begin{Proposition}\label{PSL_SL}
Let $G$ be a finitely generated group.
We have the following implications for $G$:
\begin{center}
 \begin{tikzpicture}
 \draw[-implies, double equal sign distance] (0.2,2.1) node[anchor=east] {\textit{residually $\SL_2$--finite\,\,}} -- (3,2.1)  node[anchor=west] 
 {\textit{\,\,residually $\GL_2$--finite}}
 node[anchor=south,midway] {\textit{}};
 \draw[-implies, double equal sign distance] (3,1.9) -- (0.2,1.9) node[anchor=north,midway] { $Z(G)=1$};
 \draw[implies-, double equal sign distance] (-1.25,1.7) -- (-1.25,0.75) node[anchor=north] {\hspace{-8mm} \textit{residually $\PSL_2$--finite\,\,}} node[anchor=west,midway] {\textit{}};
 \draw[-implies, double equal sign distance] (-1.5,1.7) -- (-1.5,0.75) node[anchor=east,midway] {\textit{ $Z(G)$ 2-t.f.}};
 \draw[implies-, double equal sign distance] (4.5,1.7) -- (4.5,0.75) node[anchor = north] 
 {\hspace{13mm} \textit{\hspace{-5mm} \,\,residually $\PGL_2$--finite.}} node[anchor=east,midway] {\textit{}};
 \draw[-implies, double equal sign distance] (4.75,1.7) -- (4.75,0.75) node[anchor=west,midway] {\textit{ $Z(G) = 1$}};
 \draw[-implies,double equal sign distance] (0.4,0.6) -- (3,0.6) node[anchor=south,midway] {\textit{}};
 \draw[-implies,double equal sign distance] (3,0.4) -- (0.4,0.4) node[anchor=north,midway] {\textit{}};
 \end{tikzpicture}
\end{center}
Hereby recall that  2-t.f.\ means 2-torsion-free and $Z(G)$ denotes the centre of $G$.
\end{Proposition}

Note that in passing across these implications, it may be necessary to alter the ring over which the relevant matrix group is considered when one considers the associated faithful representations. We give a simple example. The group $\SL_2(\C)$ contains a unique element of order two, and hence has no embedding of $\Z_2 \oplus \Z_2.$ Now $\Z_2 \oplus \Z_2$ embeds into $\PSL_2(\C),$ but this embedding cannot be lifted to one into $\SL_2(\C).$ 

\begin{proof}
Throughout the proof let  $G$ be a finitely generated group.
It is clear that residual $\SL_2$--finiteness and residual $\PSL_2$--finiteness imply, respectively, residual $\GL_2$--finiteness and residual $\PGL_2$--finiteness. To see that, if $G$ is 2-torsion-free, residual $\SL_2$--finiteness implies residual $\PSL_2$--finiteness, note that via Corollary~\ref{iff_faithful_rep} the former gives us a faithful representation into $\SL_2(R)$ for some $R$ and 2-torsion-freeness of $Z(G)$ implies that the image of this representation cannot contain non-identity scalar matrices. A similar proof shows that if $G$ is centreless, residual $\GL_2$--finiteness implies residual $\PGL_2$--finiteness. \\

Next, we show that  residual $\PGL_2$--finiteness implies residual $\PSL_2$--finiteness. Let $\{ g_1,\dots,g_k \}$ be a generating set for $G$; via Proposition~\ref{iff_faithful_rep}, we have a faithful $\rho\colon G \to \PGL_2(R)$ for some $R$. Choose representatives of the generators $\rho(g_1),\dots,\rho(g_k)$ of $\rho(G)$, let $a_i = \det(\rho(g_i))$ and let $R^\prime = R[x_1,\dots,x_k]/I$ where $I = (x_1^2 - a_1^{-1},\dots,x_k^2 - a_k^{-1})$. 

\begin{claim}
The obvious map $\epsilon \colon R \to R^\prime$  is injective.
\end{claim}

Recall that by definition $R[x_1,\dots,x_k]$ is the free $R$-module on the monomials $\prod x_i^{n_i}$. 
We denote by $\varphi$ the $R$-module homomorphism
\[ R[x_1,\dots,x_k] \,\,\to\,\, R\]
that is uniquely determined  by
\[ \tmprod{i=1}{m} x_i^{n_i}\,\,\mapsto \,\,\left\{\begin{array}{ll} 0,&\mbox{ if one of the $n_i$ is not  even,}\\
 \tmprod{i=1}{m} a_i^{-n_i/2},&\mbox{ if all of the $n_i$ are even.}\end{array}\right.\]
 We claim that $\varphi$ vanishes on $I$. Since $\varphi$ is $R$-linear it suffices to show that for any $j$ and any monomial $\prod x_i^{n_i}$ we have $\varphi\big((x_j^2-a_j^{-1})\prod x_i^{n_i}\big)=0$. But this follows easily from considering the two cases that the $n_i$ are all even and that one is not even separately.
It is clear that for any $r\in R$ we have $\varphi(\epsilon(r))=r$. This shows that $\epsilon$ is injective. This concludes the proof of the claim.

It follows from the claim that the $\iota\colon \PGL_2(R) \to \PGL_2(R^\prime)$ which applies the previous map $\epsilon R \to R^\prime$ to each entry is injective. This gives us a faithful representation $\iota \circ \rho\colon G \rightarrow \PGL_2(R^\prime)$. For each $i$, choosing the same representatives of the $\rho(g_i)$ as earlier we note that the representative $x_i(\iota \circ \rho)(g_i)$ has unit determinant. Thus the image of $\iota \circ \rho$ lies in the copy of $\PSL_2(R^\prime)$ inside $\PGL_2(R^\prime)$. \\

Finally we will show that  residual $\PSL_2$--finiteness implies residual $\SL_2$--finiteness; this will, using the other implications proven so far, show also that, under the same conditions, residual $\PGL_2$--finiteness implies residual $\GL_2$--finiteness. To show this, we show that, given  a representation $\rho\colon G \to \PSL_2(R)$, there exists an $R^\prime$ and a map $\varphi\colon G \rightarrow \SL_2(R^\prime)$ through which $\rho$ factors. This will complete the proof because if $G$ is residually 
$\PSL_2$--finite, it admits a faithful $\rho\colon G \to \PSL_2(R)$; this $\rho$ factors through a representation $\rho^\prime\colon G \to \SL_2(R^\prime)$ which is then also faithful and so $G$ is residually $\SL_2$--finite. The construction involved is the same as that for the $K = \PSL_2$ case in the proof of Proposition~\ref{universal_rep}. Let $G = \langle g_1,\dots,g_n \: | \: \{r_i\}_{i\in I} \rangle$, let \[S = \mathbb{Z}[x_{1a},x_{1b},x_{1c},x_{1d},\dots,x_{na},x_{nb},x_{nc},x_{nd},\{\lambda_i\}_{i\in I}]\] and then define \[p(g_i) = \left(\begin{array}{cc} x_{ia} & x_{ib} \\ x_{ic} & x_{id} \end{array}\right),\quad  p(g_i^{-1}) = \left(\begin{array}{rr} x_{id} & -x_{ib} \\ -x_{ic} & x_{ia} \end{array}\right).\] Define also $p(r_i)$  by setting that $p$ be multiplicative and then define 
\[I = \left\langle \{\det\,p(g_i)-1\}_i \cup \{\lambda_i^2-1\}_{i} \cup \{(p(r_i)-\lambda_i )_{k,l}\}_{i,k,l}\right\rangle.\]
Now set $R^\prime = S/I$ and \[\varphi\colon G \to \SL_2(R^\prime)\colon g_i \mapsto \left(\begin{array}{cc} \overline{x_{ia}} & \overline{x_{ib}} \\ \overline{x_{ic}} & \overline{x_{id}} \end{array}\right)\]
which, as it can be checked, gives a homomorphism. Now, given $\rho\colon G \to \PSL_2(R)$, let $(a^i_{kl})_{kl}$ be representatives for $\rho(g_i)$ and let $\mu_i \in R^\times$ 
be the element such 
the corresponding representative for $\rho(r_i)$ is equal to $\mu_i$ times the identity matrix. Note that  $\mu_i^2 = 1$. Define $q\colon R^\prime \to R\colon 1,x_{ia},x_{ib},x_{ic},x_{id},\lambda_i \mapsto 1,a^i_{11},a^i_{12},a^i_{21},a^i_{22},\mu_i$ and set $\psi = q_*\colon \SL_2(R^\prime) \to \SL_2(R) \to \PSL_2(R)$. Then $\rho = \psi \circ \varphi$.
\end{proof}

\begin{Corollary}\label{cor:res_PSL_all_you_need}
\begin{enumerate}
\item 
If the fundamental group of a compact 3--manifold is residually $\PSL_2$--finite or residually $\PGL_2$--finite, it is also residually $K$--finite for the other $K$.
\item If $M$ is an aspherical 3--manifold that is not a Seifert fibered manifold, then all of the above four notions of residually finiteness agree.
\end{enumerate}
\end{Corollary}

\begin{proof}
The first part follows from Proposition~\ref{PSL_SL} and the observation that all compact 3--manifold groups are finitely presented; for a proof of this latter fact, see \cite{Kirby}, where it is shown that compact topological manifolds have the homotopy type of a finite CW-complex. 

The second statement follows again from Proposition~\ref{PSL_SL} and the fact that fundamental groups of aspherical 3--manifolds are torsion-free and that the only 3--manifolds with a non-trivial center are Seifert fibered manifolds. We refer to  \cite[Theorem~2.5.5, (C.3)]{AFW} for proofs of these two statements.
\end{proof}


\section{A trip to the zoo}
\label{sec:definitions}


\subsection{Symmetric groups}

Luo conjectured that every compact 3--manifold group is residually $\PGL_2$--finite. As a first observation, recall that every compact $3$--manifold group is residually finite. See \cite{Hempel} for the case of Haken manifolds, which can be  extended to the general case via geometrisation as discussed in  \cite{Hempel} and \cite[Theorem 3.3]{Thu1982}. Correctness of Luo's conjecture would provide a list of specific finite groups which detect non-triviality. Now, as finite groups embed into symmetric groups, if we had that $S_n$, the symmetric group on $n$ letters, is residually $\PGL_2$--finite for all $n$, we would have verified Luo's conjecture. However, we have the following result, which was obtained independently in \cite{bachelorarbeit}.

\begin{Theorem}\label{thm:Sn}
\textit{$S_n$ is residually $\PGL_2$--finite if and only if $n < 5$.}
\end{Theorem}

\begin{proof}
Given positive integers $n < m$, $S_n$ embeds into $S_m$ (as the stabilizer of the final $m-n$ letters). Thus it suffices to prove that $S_4$ is residually $\PSL_2$--finite and that $S_5$ is not. The former follows from the fact that $S_4$ is isomorphic to $\PSL_2(\mathbb{F}_3)$, where $\mathbb{F}_3$ is the field with 3 elements; see \cite[Chapter 8]{Rotman} (the isomorphism arises from the faithful natural action of the latter on the projective line $\mathbb{P}^1(\mathbb{F}_3)$, which has 4 elements). For the latter, we use the following presentation for $S_5$:
\[ \left\langle x_1,x_2,x_3,x_4 \: \Bigg| \begin{array}{ll} x_i^2 = 1 & 1 \le i \le 4 \\ (x_ix_{i+1})^3 = 1 & 1 \le i < 3 \\ (x_ix_j)^2 = 1 & 1 \le i < j - 1 \le 3 \end{array} \right\rangle
\]
where $x_i$ is the transposition $(i \enspace i{+}1)$; this is a particular case of Moore's presentations for the symmetric groups, see \cite{Moore}. The required result is verified using the characterisation of residual $\PSL_2$--finiteness provided by Proposition~\ref{prop:iff_injection} above. The relevant computation is in the appendix, which shows that $S_5$ fails residual $\PSL_2$--finiteness in particular for the element $x_1x_2$. It follows from Proposition~\ref{PSL_SL} that $S_5$ also fails residual $\PGL_2$--finiteness.
\end{proof}

\begin{Remark}
The alternating groups $A_5$ and $A_6$ are residually $\PSL_2$--finite. This follows from the existence of isomorphisms $A_5 \cong \PSL_2(\mathbb{F}_4) \cong \PSL_2(\mathbb{F}_5)$ and  $A_6 \cong \PSL_2(\mathbb{F}_9)$, where $\mathbb{F}_4$, $\mathbb{F}_5$ and $\mathbb{F}_9$ are the fields with 4, 5 and 9 elements, respectively; see \cite[Chapter 8]{Rotman}. One consequence of this is that the property of being residually $\PSL_2$--finite is not inherited from finite index subgroups, even in the case of index two.
\end{Remark}


\subsection{General linear groups}

Note that $S_n$ embeds into $\GL_n(R)$ for any non-zero commutative ring with identity $R$ by mapping each permutation to the corresponding permutation matrix. Thus if we define \emph{residually $\PGL_n$--finite} in a manner similar to that in Definition~\ref{def:residual_props}, we find that $S_n$ is residually $\PGL_n$--finite. To see this, note that the canonical surjection $\GL_n(R) \to \PGL_n(R)$ is injective on the copy of $S_n$ in $\GL_n(R)$. Thus we see that Luo's conjecture holds if we weaken it to allow arbitrary dimension of matrices. On the other hand, the observation that $S_n$ embeds into $\GL_n(R)$, along with the above theorem, also gives the following:

\begin{Corollary}\label{cor:GL_res}
For any commutative ring with identity $R$, $\GL_n(R)$ is not residually $\PGL_2$--finite for $n \ge 5$.
\end{Corollary}

This raises the following question.

\begin{Question}
Let $R$ be a ring, $n\in \N$ and $k<n$. Is it possible that $\GL_n(R)$ is residually $\GL_k$--finite?
\end{Question}


\subsection{Abelian groups}

\begin{Proposition}\label{prop:fg_abelian_groups}
Every finitely generated abelian group $G$ is residually $K$--finite for all $K = \SL_2$, $\GL_2$, $\PSL_2$, $\PGL_2$.
\end{Proposition}

\begin{proof}
It is easy to see that $\Z$ and $\Z/n\Z$ are residually $K$--finite for each $K$ via matrices of the form \[ \left( \begin{array}{cc} 1 & n \\ 0 & 1 \end{array} \right) \]
where the ring of entries is $\Z$ in the case of $\Z$ and $\Z/n\Z$ in the case of $\Z/n\Z$. Now, given a finitely generated abelian group $G$, decompose $G$ via the classification theorem for finitely generated abelian groups and then use the projections onto each factor.
\end{proof}

Thus, given a finitely generated group $G$ and an element $g \in G$ that is not contained in the commutator subgroup $[G,G]$, by passing to the abelianisation, we can construct a finite commutative ring $R$ and a homomorphism $G \to K(R)$ that does not kill $g$. As such, it is only elements in the commutator subgroup that we ever need to worry about.


\subsection{Dihedral groups}

Denote $D_{2k} = \langle a, b \mid a^{k} = b^2 = 1, bab = a^{-1}\rangle$ the dihedral group of order $2k.$ 
A faithful representation $D_{2k}\to \PSL_2(\C)$ is defined by 
$$a \mapsto \pm \begin{pmatrix} \xi & 0 \\ 0 & \xi^{-1}\end{pmatrix} \quad\text{and}\quad b \mapsto \pm \begin{pmatrix} \hfill 0 & 1 \\ -1 & 0 \end{pmatrix},$$
where $\xi=\exp(\pi i /k).$ Whence $D_{2k}$ is residually $K$--finite for all $K = \SL_2$, $\GL_2$, $\PSL_2$, $\PGL_2$. This family includes examples with 2-torsion in their centre.


\subsection{Surface groups}

\begin{Lemma}\label{lem:surface groups}
The fundamental group of an orientable, compact, connected surface is residually $\PSL_2$--finite.
\end{Lemma}

\begin{proof}
Let $S_{g,b}$, or $S_g$ if $b=0$, denote our surface, where $g$ denotes the genus and $b$ the number of boundary components. If $b > 0$, $\pi_1(S_{g,b})$ is a free group on $2g+b-1$ generators. It is well-known that a free group of at most countable rank embeds into the free group on two generators and that the free group on two generators embeds into $\SL_2(\Z)$ as the subgroup generated by
\[
 \left(\begin{array}{cc} 1 & 2 \\ 0 & 1 \end{array}\right) \hspace{0.5cm} \text{and} \hspace{0.5cm} \left( \begin{array}{cc} 1 & 0 \\ 2 & 1\end{array}\right).
\]
Because free groups are torsion-free, upon application of Propositions~\ref{iff_faithful_rep} and \ref{PSL_SL}, we have the result. Now suppose that $b=0$; that is, $S_{g,b}$ is closed. If $g = 0$, we have the 2-sphere which has trivial fundamental group and there is nothing to show. If $g = 1$, we have the 2-torus which has an abelian fundamental group and we have already dealt with the case of finitely generated abelian groups in Section 2.  Next, it is shown in \cite{Newman} that, for $g \ge 2$, $\pi_1(S_g)$ embeds into $\pi_1(S_2)$. It is shown in \cite{Magnus} that $\pi_1(S_2)$ embeds into $\SL_2(\C)$ and using torsion-freeness of closed surface groups and Propositions~\ref{iff_faithful_rep} and \ref{PSL_SL}, we have the result.
\end{proof}


\subsection{The integral Heisenberg group}

The \emph{integral Heisenberg group} $H$ has the well-known presentation $\langle a,b,c \: | \: [a,b] = c, c \ \: \text{central} \rangle.$ We now show that it is residually $\PSL_2$--finite.

An element of $H$ is a word $w$ in the letters $a,b,c$. Because $c$ is central, one can write $w = w^\prime c^{q}$ for some $q \in \Z$ and word $w^\prime$ in $a,b$. Because $[a,b] = c$ and so $ab = c(ba)$, one can interchange $a,b$ in $w^\prime$ at the cost of introducing a $c$ which can once again be pushed off to the right so that one can write $w = a^m b^n c^p$ for some $m,n,p \in \Z$. Note that the abelianisation of $H$ is $\Z \times \Z$, generated by $A = \overline{a}, B= \overline{b}$. The element $w$ maps to $mA + nB$ and so, since $\Z \times \Z$ is residually $\PSL_2$--finite, we see that we have reduced to the case that $w$ is power of $c$. Let $R = \C[x,y,z]/(x(1-y^2)^2,yz-1)$ and define $\rho\colon H \to \SL_2(R)$ as follows
\[
a \mapsto \left(\begin{array}{cc} 1 & \overline{x} \\ 0 & 1 \end{array}\right) \:\:\:
b \mapsto \left(\begin{array}{cc} \overline{y} & 0 \\ 0 & \overline{z} \end{array}\right) \:\:\:
c \mapsto \left(\begin{array}{cc} 1 & \overline{x(1-y^2)} \\ 0 & 1 \end{array}\right).
\]
It can be checked that the relations of $H$ are indeed satisfied and also that $x(1-y^2) \notin I$ (to see a verification via SageMath, see the appendix) so that $\rho(c) \neq 1$. Similarly, as 
\[
\left(\begin{array}{cc} 1 & x(1-y^2) \\ 0 & 1 \end{array}\right)^n =
\left(\begin{array}{cc} 1 & nx(1-y^2) \\ 0 & 1 \end{array}\right)
\]
we see that $\rho(c^n) \neq 1$ for all $n \neq 0$. We now projectivize, noting that this does not kill $\rho(c^n)$ because an off-diagonal entry is non-zero, and then apply Proposition~\ref{prop:reducing_the_ring}.


\subsection{Hyperbolic manifolds and trivial product geometries}

\begin{Proposition}\label{prop:geometric_manifolds}
If the compact orientable geometric 3--manifold $M$ is modelled on $\Hyp^3$, $\Hyp^2 \times \E$ or $\Sph^2 \times \E,$ then $\pi_1 M$ is residually $\SL_2,\PSL_2, \GL_2, \PGL_2$--finite.
\end{Proposition}

\begin{proof}
If $M$ is modelled on $\Hyp^3$, the holonomy representation gives an embedding of $\pi_1(M)$ into $\text{Isom}^+(\Hyp^3) \cong \PSL(2,\C)$. Thus $\pi_1(M)$ is residually $\PSL_2$--finite and Proposition~\ref{PSL_SL} completes the result. Note here that another way to see that $\pi_1(M)$ is also residually $\SL_2$--finite is via a well-known result due to Thurston that we can lift the holonomy representation into $\PSL(2,\C)$ to one into $\SL(2,\C)$; see \cite{Shalen}.

If $M$ is modelled on $\Hyp^2 \times \E$, the holonomy representation gives an embedding of $\pi_1(M)$ into $\text{Isom}^{+}(\Hyp^2 \times \E^1) = \text{Isom}^{+}(\Hyp^2) \times \text{Isom}^{+}(\E^1) \cong \PSL(2,\R) \times \R$. Note that $\R$ embeds into $\PSL(2,\R)$, exactly via the usual matrix representation for $v \in \text{Isom}^{+}(\E^1)$ as \[\left(\begin{array}{cc} 1 & v \\ 0 & 1 \end{array} \right).\] Now we post-compose the above embedding with the projections of $\PSL_2(\R) \times \R$ onto either factor and then use Proposition~\ref{prop:reducing_the_ring}. This gives us that $\pi_1(M)$ is residually $\PSL_2$--finite and Corollary~\ref{cor:res_PSL_all_you_need} does the rest of the work. \\

If $M$ is modelled on $\Sph^2 \times \E$ and orientable, the holonomy representation gives an embedding of $\pi_1(M)$ into $\text{Isom}^+(\Sph^2) \times \text{Isom}^+(\E) \cong \text{SO}(3,\R) \times \R$. Again, due to the presence of the projections, we need only worry about the two factors in the product, and we can deal with the $\R$ factor as we did in the previous case. To deal with the $\text{SO}(3,\R)$ factor, we recall the well-known double covering $\text{SU}(2,\C) \to \text{SO}(3,\R)$ which gives an isomorphism $\text{SO}(3,\R) \cong \text{PSU}(2,\C) \le \PSL_2(\C)$ and so using Proposition~\ref{prop:reducing_the_ring} and Corollary~\ref{cor:res_PSL_all_you_need}, we conclude that $\pi_1(M)$ satisfies the conclusion.
\end{proof}


\subsection{A counterexample to Luo's conjecture}
\label{sec:counterexample}

Let $M$ be the $(4,1)$-Dehn filling, using the knot theoretic framing, of the figure-8 knot complement. We will show that $M$ is a  counterexample to Luo's conjecture. In SnapPy~\cite{SnapPy} one can construct a triangulation of $M$ and this triangulation can then be imported into Regina~\cite{Regina}. See the appendix. Regina then gives the following presentation for $\Gamma = \pi_1(M)$:
\[
\Gamma = \langle a,b \: | \: a^{-1}b^2a^{-3}b^2 = 1, ba^{-2}ba^{-2}b^3a^{-2} = 1 \rangle.
\]
We re-write this presentation by making the substitutions $a \leadsto b^{-1}$, $b \leadsto a^{-1}$ and set $c = b^2a^{-2}$; this leads to the following presentation:
\[
\Gamma = \langle a,b,c \: | \: ca^2 = b^2, c^{-1}b=bc, ac^{-1}a^{-1} = cac\rangle.
\]
This can be re-written as
\[
\langle a,b,c \: | \: c = b^2a^{-2}, \underbrace{1 = bcb^{-1}c}_{\text{Klein bottle}},\overbrace{a^2=(ac)^3}^{\text{trefoil complement}} \rangle
\]
which highlights the presence of a trefoil knot complement and a Klein bottle; in fact, $M$ can be constructed as the identification of a trefoil knot complement and a twisted $I$-bundle over a Klein bottle. Letting $\Gamma_1 = \langle u,v  \: |  \: u^3 = v^2 \rangle$ and $\Gamma_2 = \langle j,k  \: |  \: jkj^{-1}k  = 1 \rangle$, the peripheral subgroups $\langle v^{-1}u,u^3\rangle \cong \Z \oplus \Z$ and $\langle k,j^2 \rangle \cong \Z \oplus \Z$ are glued via the identifications $v^{-1}u \leftrightarrow k$, $u^3 \leftrightarrow k^{-1}j^2$.

We return to the second presentation for $\Gamma$ above and construct the universal representation of 
Proposition~\ref{universal_rep}. This gives
$\varphi_{\SL_2}\co \Gamma \to \SL_2(S_{\SL_2}/I_{\SL_2})$, where $S_{\SL_2} = \Z[i,j,k,l,p,q,r,s,w,x,y,z]$,
\[
a \mapsto \left( \begin{array}{cc} i & j \\ k & l \end{array} \right)
\qquad
b \mapsto \left( \begin{array}{cc} p & q \\ r & s \end{array} \right)
\qquad
c \mapsto \left( \begin{array}{cc} w & x \\ y & z \end{array} \right)
\]
and the ideal $I_{\SL_2}$ is generated by $il-kj-1,ps-rq-1,wz-yx-1,$ and 12 equations arising from the relations. We have
\begin{align*}
\left( \begin{array}{cc} p & q \\ r & s \end{array} \right)^4  &= 
\left( \begin{array}{cc} (p^2+qr)^2+qr(p+s)(p+s) & q(p^2+qr)(p+s)+q(p+s)(qr+s^2) \\ r(p+s)(p^2+qr)+r(qr+s^2)(p+s) & qr(p+s)(p+s)+(qr+s^2)^2 \end{array} \right) \\ 
&=: \left( \begin{array}{cc} f_1 & f_2 \\ f_3 & f_4 \end{array} \right).
\end{align*}
It can be verified via SageMath~\cite{Sage}, that $f_1-1,f_2,f_3,f_4-1 \in I_{\SL_2}$ so that $\varphi_{\SL_2}(b^4) = 1$ and hence $\varphi_{\SL_2}$ is not injective. See the appendix. Thus $b^4$ is killed in any representation $\Gamma \to \SL_2(R)$, $\Gamma$ is not residually $\SL_2$--finite and so, using the argument of the proof of Proposition~\ref{PSL_SL}, is also not residually $\PSL_2$--finite.



\subsection{Quaternionic space}
\label{sec:quaternionic}

This example illustrates the use of Luo's construction to obtain positive results using triangulations.
Figure~\ref{fig:quaternionic} below depicts an oriented triangulation of quaternionic space $S^3/Q_8$ from Regina, \cite{Regina}; the orientations on the simplices here are $v_i \to v_{i+1}$ and $v_i^\prime \to v_{i+1}^\prime$. The action of $Q_8$ on $S^3$ is the natural one after identifying $S^3$ with $\{(z,w) \in \C^2 \: | \: |z|^2 + |w|^2 = 1\}$ and $Q_8$ with a subgroup of $\SL_2(\C)$ via
\[
 1 \mapsto \left(\begin{array}{cc}
  1 & 0 \\
  0 & 1
 \end{array}\right) \hspace{0.5cm}
 i \mapsto \left(\begin{array}{rr}
  i & 0 \\
  0 & -i
 \end{array}\right) \hspace{0.5cm}
 j \mapsto \left(\begin{array}{rr}
  0 & 1 \\
  -1 & 0
 \end{array}\right) \hspace{0.5cm}
 k \mapsto \left(\begin{array}{cc}
  0 & i \\
  i & 0
 \end{array}\right).
\]

%

The face-pairings, specified via the vertices $v_i, v_i^\prime$, are as follows:
\[
\varphi_1\colon v_0,v_1,v_2 \mapsto v_3^\prime,v_0^\prime,v_1^\prime
\hspace{1cm}
\varphi_2\colon v_0,v_1,v_3 \mapsto v_1^\prime,v_2^\prime,v_0^\prime
\]
\[
\varphi_3\colon v_0,v_2,v_3 \mapsto v_2^\prime,v_0^\prime,v_3^\prime
\hspace{1cm}
\varphi_4\colon v_1,v_2,v_3 \mapsto v_3^\prime,v_2^\prime,v_1^\prime.
\]
From these one can compute the edge cycles, i.e.\ the cyclic sequences of edges identified to one another; these are indicated by colour in Figure~\ref{fig:quaternionic}. We search for solutions to Thurston's equations by labelling the edges of our triangulation by arbitrary elements of some ring $R$ as in Figure~\ref{fig:quaternionicThurston}.

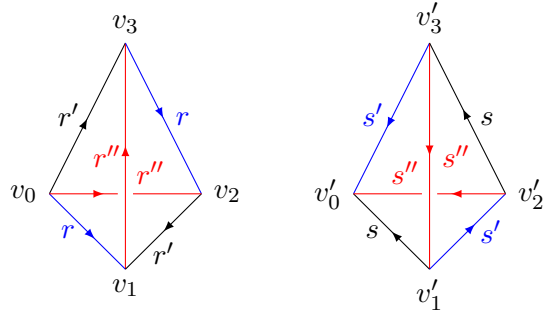
\begin{figure}[h]
\centering
\tikzset{middlearrow/.style={
        decoration={markings,
            mark= at position 0.5 with {\arrow{#1}} ,
        },
        postaction={decorate}
    }
}

\tikzset{backarrow/.style={
        decoration={markings,
            mark= at position 0.8 with {\arrow{#1}} ,
        },
        postaction={decorate}
    }
}

\begin{tikzpicture}[scale=1]
 \draw[middlearrow={latex}] (0,0) to (1,2) node[black,anchor=south]{$v_3$};
 \draw[middlearrow={latex reversed},red] (1,2) -- (1,-1) node[midway,anchor=east,xshift=1mm]{$r^{\prime\prime}$} node[black,anchor=north]{$v_1$};
 \draw[middlearrow={latex reversed},blue] (1,-1) -- (0,0) node[midway,anchor=east]{$r$} node[black,anchor=east]{$v_0$};
 \draw[middlearrow={latex},blue] (1,2) -- (2,0) node[midway,anchor=west]{$r$} node[black,anchor=west]{$v_2$};
 \draw[middlearrow={latex}] (2,0) -- (1,-1) node[midway,anchor=north]{$r^\prime$};
 \draw[backarrow={latex},red] (0,0) -- (0.9,0);
 \draw[red] (1.1,0) -- (2,0) node[near start,anchor=south]{$r^{\prime\prime}$};
 \node at (0.25,1.05) {$r^\prime$};
 
 \draw[middlearrow={latex reversed},blue] (4,0) to (5,2) node[black,anchor=south]{$v^\prime_3$};
 \draw[middlearrow={latex},red] (5,2) -- (5,-1) node[midway,anchor=west]{$s^{\prime\prime}$} node[black,anchor=north]{$v^\prime_1$};
 \draw[middlearrow={latex}] (5,-1) -- (4,0) node[midway,anchor=east]{$s$} node[black,anchor=east]{$v^\prime_0$};
 \draw[middlearrow={latex reversed}] (5,2) -- (6,0) node[midway,anchor=west]{$s$} node[black,anchor=west]{$v^\prime_2$};
 \draw[middlearrow={latex reversed},blue] (6,0) -- (5,-1) node[midway,anchor=west]{$s^\prime$};
 \draw[red] (4,0) -- (4.9,0) node[near end,anchor=south]{$s^{\prime\prime}$};
 \draw[backarrow={latex},red] (6,0) -- (5.1,0);
 \node[blue] at (4.25,1.05) {$s^\prime$};
\end{tikzpicture}
\caption{Triangulation of quaternionic space $S^3/Q_8$ from Regina -- ``SFS [S2: (2,1) (2,1) (2,-1)]: \#1'' in ``Closed Census (Orientable)''}
\label{fig:quaternionic}
\label{fig:quaternionicThurston}
\end{figure}
Following the edge cycles, the gluing equations are then
\[
r^2(s^\prime)^2 = 1 \hspace{1cm} (r^\prime)^2s^2 = 1 \hspace{1cm} (r^{\prime\prime})^2(s^{\prime\prime})^2 = 1.
\]
Combining these with the parameter relations, one finds that the parameter relations together with
\[
r^2(s^\prime)^2 = 1 \hspace{1cm} 2(rs^\prime-1) = 0
\]
are necessary and sufficient conditions on the labels (to see this, re-write the gluing equations in terms of $r$ and $s^\prime$ alone). Thus in the case of an $R$ in which $2$ is not a zero divisor, in particular $\C$, we have that $s^\prime = r^{-1}$ and that the following, for $r \neq 0,1$, are all the solutions to Thurston's equations:
\begin{equation}
(r,r^\prime,r^{\prime\prime},s,s^\prime,s^{\prime\prime}) = \left(r,\smfrac{1}{1-r},\smfrac{r-1}{r},1-r,\smfrac{1}{r},\smfrac{r}{r-1}\right).
\end{equation}
Allowing ourselves extra flexibility, there are other possible solutions. For example, it can be checked that setting $R = \mathbb{F}_4[x]/(x^2)$, where $\mathbb{F}_4 = \{0,1,a,b\}$ is the field with four elements, and $r = a, s^\prime = b + x$, gives a solution to Thurston's equations, namely
\begin{equation}
(r,r^\prime,r^{\prime\prime},s,s^\prime,s^{\prime\prime}) = (a,a,a,b+bx,b+x,b+ax).
\end{equation}
Now we compute the associated holonomy representations. First, consider the generic solution (1) over $\C$, for $r = z$. As in \cite{Feng}, we build a corresponding solution to the homogeneous Thurston equations and then build the holonomy representation $\rho$ as follows:
\begin{itemize}
	\item For the solution to the homogeneous Thurston equations, we choose $q_0 = \{\{v_0,v_1\},\{v_2,v_3\}\}$ and $q_0^\prime = \{\{v^\prime_0,v^\prime_1\},\{v^\prime_2,v^\prime_3\}\}$ as the initial normal quadrilaterals.
	\item We label $v_0,v_1,v_2$ by $[1,0]^t$, $[0,1]^t$, $[1,1]^t$.
	\item We extend this labeling across $\varphi_1$ and via the solution to the homogeneous Thurston equations.
\end{itemize}
Doing so, we find that the resulting labels for $v_3$, $v_0^\prime$, $v_1^\prime$, $v_2^\prime$ and $v_3^\prime$ are, respectively, $[1,z]^t$, $[0,1]^t$, $[1,1]^t$, $[1,z]^t$, $[1,0]^t$. The images of the generators $\varphi_2$, $\varphi_3$, $\varphi_4$, or more precisely the element of $\pi_1(S^3/Q_8) \cong Q_8$ which they represent, are then as follows:
\begin{align*}
&\rho(\varphi_2)\colon \left[\begin{array}{c} 1 \\ 0 \end{array}\right],\left[\begin{array}{c} 0 \\ 1 \end{array}\right],\left[\begin{array}{c} 1 \\ z \end{array}\right] \mapsto \left[\begin{array}{c} 1 \\ 1 \end{array}\right],\left[\begin{array}{c} 1 \\ z \end{array}\right],\left[\begin{array}{c} 0 \\ 1 \end{array}\right] \leadsto \rho(\varphi_2) = \left[\begin{array}{cc} z & -1 \\ z & -z \end{array}\right] \\
&\rho(\varphi_3)\colon \left[\begin{array}{c} 1 \\ 0 \end{array}\right],\left[\begin{array}{c} 1 \\ 1 \end{array}\right],\left[\begin{array}{c} 1 \\ z \end{array}\right] \mapsto \left[\begin{array}{c} 1 \\ z \end{array}\right],\left[\begin{array}{c} 0 \\ 1 \end{array}\right],\left[\begin{array}{c} 1 \\ 0 \end{array}\right] \leadsto \rho(\varphi_4) = \left[\begin{array}{cc} 1 & -1 \\ z & -1 \end{array}\right] \\
&\rho(\varphi_4)\colon \left[\begin{array}{c} 0 \\ 1 \end{array}\right],\left[\begin{array}{c} 1 \\ 1 \end{array}\right],\left[\begin{array}{c} 1 \\ z \end{array}\right] \mapsto \left[\begin{array}{c} 1 \\ 0 \end{array}\right],\left[\begin{array}{c} 1 \\ z \end{array}\right],\left[\begin{array}{c} 1 \\ 1 \end{array}\right] \leadsto \rho(\varphi_4) = \left[\begin{array}{cc} 0 & 1 \\ z & 0 \end{array}\right].
\end{align*}
It is clear that $\rho(\varphi_2), \rho(\varphi_3), \rho(\varphi_4)$ are pairwise distinct and one can check that any two of these (in either order) multiply to give the third. Thus, for any $z \neq 0, 1$, the image of the holonomy representation is the Klein-4 group. \\

Consider now the solution (2) over $\mathbb{F}_4[x]/(x^2)$. We will see that we can achieve a larger image by not working over $\C$ and using this labelling. Repeating the above procedure, we find that for $v_0, v_1, v_2, v_3, v_0^\prime, v_1^\prime, v_2^\prime, v_3^\prime$ receive the labels $[1,0]^t, [0,1]^t, [1,1]^t, [1,a]^t, [0,1]^t, [1,1]^t, [1,a+bx]^t, [1,0]^t$, respectively. The associated holonomy representation $\rho^\prime$ is generated by:
\begin{align*}
&\rho^\prime(\varphi_2)\colon \left[\begin{array}{c} 1 \\ 0 \end{array}\right],\left[\begin{array}{c} 0 \\ 1 \end{array}\right],\left[\begin{array}{c} 1 \\ a \end{array}\right] \mapsto \left[\begin{array}{c} 1 \\ 1 \end{array}\right],\left[\begin{array}{c} 1 \\ a+bx \end{array}\right],\left[\begin{array}{c} 0 \\ 1 \end{array}\right] \leadsto \rho^\prime(\varphi_2) = \left[\begin{array}{cc} a & 1 \\ a & a+bx \end{array}\right] \\
&\rho^\prime(\varphi_3)\colon \left[\begin{array}{c} 1 \\ 0 \end{array}\right],\left[\begin{array}{c} 1 \\ 1 \end{array}\right],\left[\begin{array}{c} 1 \\ a \end{array}\right] \mapsto \left[\begin{array}{c} 1 \\ a+bx \end{array}\right],\left[\begin{array}{c} 0 \\ 1 \end{array}\right],\left[\begin{array}{c} 1 \\ 0 \end{array}\right] \leadsto \rho^\prime(\varphi_3) = \left[\begin{array}{cc} 1 & 1 \\ a+bx & 1+ax \end{array}\right] \\
&\rho^\prime(\varphi_4)\colon \left[\begin{array}{c} 0 \\ 1 \end{array}\right],\left[\begin{array}{c} 1 \\ 1 \end{array}\right],\left[\begin{array}{c} 1 \\ a \end{array}\right] \mapsto \left[\begin{array}{c} 1 \\ 0 \end{array}\right],\left[\begin{array}{c} 1 \\ a+bx \end{array}\right],\left[\begin{array}{c} 1 \\ 1 \end{array}\right] \leadsto \rho^\prime(\varphi_4) = \left[\begin{array}{cc} x & 1+x \\ a+bx & 0 \end{array}\right].
\end{align*}
It can now be checked that 
\[
\rho^\prime(\varphi_2)^2 = \rho^\prime(\varphi_3)^2 = \rho^\prime(\varphi_4)^2 = \left[\begin{array}{cc} 1 & bx \\ x & 1 \end{array}\right]
\]
and if we denote this common square $J$, that
\[
J^2 = 1 \hspace{1cm} \rho^\prime(\varphi_2)\rho^\prime(\varphi_3)\rho^\prime(\varphi_4) = J.
\]
It follows that this holonomy representation is faithful with image isomorphic to $Q_8$, where an explicit isomorphism is given by $J,\rho^\prime(\varphi_2),\rho^\prime(\varphi_3),\rho^\prime(\varphi_4) \mapsto -1,i,j,k$.



\address{Universit\"at Regensburg, Fakult\"at f\"ur Mathematik, 93053 Regensburg, Germany} 
\address{Department of Mathematics, University of Michigan, Ann Arbor, MI 48109-1043, USA} 
\address{School of Mathematics and Statistics, The University of Sydney, NSW 2006, Australia} 
\email{sfriedl@gmail.com} 
\email{montekg@umich.edu}
\email{stephan.tillmann@sydney.edu.au} 

\Addresses


\section*{Appendix}

\textbf{Ideal membership check for $S_5$.} The following is SageMath, \cite{Sage}, code verifying the claim in the proof of Theorem~\ref{thm:Sn} that $S_5$ is not residually $\PSL_2$--finite. Here we map the generators as follows: \[x_1 \mapsto \left(\begin{array}{cc}a & b\\ c&d\end{array}\right) \hspace{1cm} x_2 \mapsto \left(\begin{array}{cc} i & j\\ k&l\end{array}\right) \hspace{1cm} x_3 \mapsto \left(\begin{array}{cc} p& q\\ r&s\end{array}\right) \hspace{1cm} x_4 \mapsto \left(\begin{array}{cc} w & x\\ y&z\end{array}\right). \]

\begin{lstlisting}[style=base, mathescape=true, escapeinside={<!}{!>}]
R = PolynomialRing(ZZ,16,'abcdijklpqrswxyz')

R

	@Multivariate Polynomial Ring in a, b, c, d, i, j, k, l, p, q, r, s, w, x, y, z over Integer Ring@
	
a,b,c,d,i,j,k,l,p,q,r,s,w,x,y,z = R.gens()

<!\textcolor{gray}{\# The determinant polynomials}!>
f1 = a*d-b*c-1
f2 = i*l-k*j-1
f3 = p*s-r*q-1
f4 = w*z-y*x-1

<!\textcolor{gray}{\# The polynomials corresponding to the relations $x_i^2 = 1$}!>
f5 = a^2 + b*c-1
f6 = b*(a + d)
f7 = c*(a + d)
f8 = d^2 + b*c-1
f9 = i^2 + j*k-1
f10 = j*(i + l)
f11 = k*(i + l)
f12 = l^2 + j*k-1
f13 = p^2 + q*r-1
f14 = q*(p + s)
f15 = r*(p + s)
f16 = s^2 + q*r-1
f17 = w^2 + x*y-1
f18 = x*(w + z)
f19 = y*(w + z)
f20 = z^2 + x*y-1

<!\textcolor{gray}{\# The polynomials corresponding to the relations $(x_ix_{i+1})^3 = 1$}!>
f21 = i*(a*(i*(c*(a*j + b*l) + a*(b*k + a*i)) + k*(d*(a*j + b*l) + b*(b*k + a*i))) + c*(j*(c*(a*j + b*l) + a*(b*k + a*i)) + l*(d*(a*j + b*l) + b*(b*k + a*i)))) + k*(b*(i*(c*(a*j + b*l) + a*(b*k + a*i)) + k*(d*(a*j + b*l) + b*(b*k + a*i))) + d*(j*(c*(a*j + b*l) + a*(b*k + a*i)) + l*(d*(a*j + b*l) + b*(b*k + a*i))))-1
f22 = j*(a*(i*(c*(a*j + b*l) + a*(b*k + a*i)) + k*(d*(a*j + b*l) + b*(b*k + a*i))) + c*(j*(c*(a*j + b*l) + a*(b*k + a*i)) + l*(d*(a*j + b*l) + b*(b*k + a*i)))) + l*(b*(i*(c*(a*j + b*l) + a*(b*k + a*i)) + k*(d*(a*j + b*l) + b*(b*k + a*i))) + d*(j*(c*(a*j + b*l) + a*(b*k + a*i)) + l*(d*(a*j + b*l) + b*(b*k + a*i))))
f23 = i*(a*(i*(c*(c*j + d*l) + a*(d*k + c*i)) + k*(d*(c*j + d*l) + b*(d*k + c*i))) + c*(j*(c*(c*j + d*l) + a*(d*k + c*i)) + l*(d*(c*j + d*l) + b*(d*k + c*i)))) + k*(b*(i*(c*(c*j + d*l) + a*(d*k + c*i)) + k*(d*(c*j + d*l) + b*(d*k + c*i))) + d*(j*(c*(c*j + d*l) + a*(d*k + c*i)) + l*(d*(c*j + d*l) + b*(d*k + c*i))))
f24 = j*(a*(i*(c*(c*j + d*l) + a*(d*k + c*i)) + k*(d*(c*j + d*l) + b*(d*k + c*i))) + c*(j*(c*(c*j + d*l) + a*(d*k + c*i)) + l*(d*(c*j + d*l) + b*(d*k + c*i)))) + l*(b*(i*(c*(c*j + d*l) + a*(d*k + c*i)) + k*(d*(c*j + d*l) + b*(d*k + c*i))) + d*(j*(c*(c*j + d*l) + a*(d*k + c*i)) + l*(d*(c*j + d*l) + b*(d*k + c*i))))-1
f29 = w*(p*(w*(p*(p*w + q*y) + r*(p*x + q*z)) + y*(q*(p*w + q*y) + s*(p*x + q*z))) + r*(x*(p*(p*w + q*y) + r*(p*x + q*z)) + z*(q*(p*w + q*y) + s*(p*x + q*z)))) + y*(q*(w*(p*(p*w + q*y) + r*(p*x + q*z)) + y*(q*(p*w + q*y) + s*(p*x + q*z))) + s*(x*(p*(p*w + q*y) + r*(p*x + q*z)) + z*(q*(p*w + q*y) + s*(p*x + q*z))))-1
f30 = x*(p*(w*(p*(p*w + q*y) + r*(p*x + q*z)) + y*(q*(p*w + q*y) + s*(p*x + q*z))) + r*(x*(p*(p*w + q*y) + r*(p*x + q*z)) + z*(q*(p*w + q*y) + s*(p*x + q*z)))) + z*(q*(w*(p*(p*w + q*y) + r*(p*x + q*z)) + y*(q*(p*w + q*y) + s*(p*x + q*z))) + s*(x*(p*(p*w + q*y) + r*(p*x + q*z)) + z*(q*(p*w + q*y) + s*(p*x + q*z))))
f31 = w*(p*(w*(p*(r*w + s*y) + r*(r*x + s*z)) + y*(q*(r*w + s*y) + s*(r*x + s*z))) + r*(x*(p*(r*w + s*y) + r*(r*x + s*z)) + z*(q*(r*w + s*y) + s*(r*x + s*z)))) + y*(q*(w*(p*(r*w + s*y) + r*(r*x + s*z)) + y*(q*(r*w + s*y) + s*(r*x + s*z))) + s*(x*(p*(r*w + s*y) + r*(r*x + s*z)) + z*(q*(r*w + s*y) + s*(r*x + s*z))))
f32 = x*(p*(w*(p*(r*w + s*y) + r*(r*x + s*z)) + y*(q*(r*w + s*y) + s*(r*x + s*z))) + r*(x*(p*(r*w + s*y) + r*(r*x + s*z)) + z*(q*(r*w + s*y) + s*(r*x + s*z)))) + z*(q*(w*(p*(r*w + s*y) + r*(r*x + s*z)) + y*(q*(r*w + s*y) + s*(r*x + s*z))) + s*(x*(p*(r*w + s*y) + r*(r*x + s*z)) + z*(q*(r*w + s*y) + s*(r*x + s*z))))-1

<!\textcolor{gray}{\# The polynomials corresponding to the relations $(x_ix_j)^2 = 1$}!>
f33 = p*(a*(a*p + b*r) + c*(a*q + b*s)) + r*(b*(a*p + b*r) + d*(a*q + b*s))-1
f34 = q*(a*(a*p + b*r) + c*(a*q + b*s)) + s*(b*(a*p + b*r) + d*(a*q + b*s))
f35 = p*(a*(c*p + d*r) + c*(c*q + d*s)) + r*(b*(c*p + d*r) + d*(c*q + d*s))
f36 = q*(a*(c*p + d*r) + c*(c*q + d*s)) + s*(b*(c*p + d*r) + d*(c*q + d*s))-1
f37 = w*(a*(a*w + b*y) + c*(a*x + b*z)) + y*(b*(a*w + b*y) + d*(a*x + b*z))-1
f38 = x*(a*(a*w + b*y) + c*(a*x + b*z)) + z*(b*(a*w + b*y) + d*(a*x + b*z))
f39 = w*(a*(c*w + d*y) + c*(c*x + d*z)) + y*(b*(c*w + d*y) + d*(c*x + d*z))
f40 = x*(a*(c*w + d*y) + c*(c*x + d*z)) + z*(b*(c*w + d*y) + d*(c*x + d*z))-1
f41 = w*(i*(i*w + j*y) + k*(i*x + j*z)) + y*(j*(i*w + j*y) + l*(i*x + j*z))-1
f42 = x*(i*(i*w + j*y) + k*(i*x + j*z)) + z*(j*(i*w + j*y) + l*(i*x + j*z))
f43 = w*(i*(k*w + l*y) + k*(k*x + l*z)) + y*(j*(k*w + l*y) + l*(k*x + l*z))
f44 = x*(i*(k*w + l*y) + k*(k*x + l*z)) + z*(j*(k*w + l*y) + l*(k*x + l*z))-1

I = (f1,f2,f3,f4,f5,f6,f7,f8,f9,f10,f11,f12,f13,f14,f15,f16,f17,f18,f19,f20,21,f22,f23,
f24,f25,f26,f27,f28,f29,f30,f31,f32,f33,f34,f35,f36,f37,f38,f39,f40,f41,f42,f43,f44)*R

a*i+b*k-1 in I

	@True@
	
a*j+b*l in I

	@True@

c*i+d*k in I

	@True@

c*j+d*l-1 in I

	@True@
\end{lstlisting}

\noindent \textbf{Ideal membership check for the case of the Heisenberg group.} The following is SageMath, \cite{Sage}, code verifying the claim that $x(1-y^2) \notin I$ where $I$ is the ideal $(x(1-y^2)^2,yz-1)$ in $\C[x,y,z]$.

\begin{lstlisting}[style=base]
R = PolynomialRing(CC,3,`xyz')

R

	@Multivariate Polynomial Ring in x,y,z over Complex Field with 53 bits of precision@
	
x,y,z = R.gens()

I = (x(1-y^2)^2,yz-1)*R

I

	@Ideal (x*y^4 + (-2.00000000000000)*x*y^2+x,y*z-1.00000000000000) of Multivariate Polynomial Ring in x,y,z over Complex Field with 53 bits of precision@

x(1-y^2) in I

	@False
\end{lstlisting}

\noindent \textbf{Triangulation of the $(4,1)$-Dehn filling of the figure-8 knot complement.}The following is SnapPy, \cite{SnapPy}, code which will produce a triangulation of the $(4,1)$-Dehn filling, using the knot theoretic framing, of the figure-8 knot complement and save it to a file entitled ``fig-eight-4-1.tri'' which can be imported into Regina.

\begin{lstlisting}[style=base]
@In@ @[@ @@1@@ @]@: M = Manifold('4_1')

@In@ @[@ @@2@@ @]@: M.dehn_fill( (4,1) )

@In@ @[@ @@3@@ @]@: N = M.filled_triangulation()

@In@ @[@ @@4@@ @]@: N.save('fig-eight-4-1.tri')
\end{lstlisting}

To import this triangulation into Regina, \cite{Regina}, follow: File --$>$ Import --$>$ SnapPea triangulation. \\

\textbf{Ideal membership check for the universal representation in our counterexample.} The following is SageMath, \cite{Sage}, code verifying the claim about the element $b^4$ in Section 4.3. Here the polynomials $f1,f2,f3$ arise as determinants, $f4,f5,f6,f7$ from the relation $ca^2 = b^2$, $f8,f9,f10,f11$ from the relation $c^{-1}b=bc$ and $f12,f13,f14,f15$ from the relation $ac^{-1}a^{-1}=cac$.

\begin{lstlisting}[style=base, mathescape = true, escapeinside = {<!}{!>}]
R = PolynomialRing(ZZ,12,`ijklpqrsxywz')

R

	@Multivariate Polynomial Ring in i,j,k,l,p,q,r,s,w,x,y,z over Integer Ring@
	
i,j,k,l,p,q,r,s,w,x,y,z = R.gens()

<!\textcolor{gray}{\# The determinant polynomials}!>
f1 = i*l-k*j-1
f2 = p*s-r*q-1
f3 = w*z-y*x-1

<!\textcolor{gray}{\# The polynomials corresponding to $ca^2 = b^2$}!>
f4 = w*i^2+w*j*k+x*i*k+x*l*k-p^2-q*r
f5 = w*i*j+w*l*j+x*j*k+x*l^2-p*q-s*q
f6 = y*i^2+y*j*k+z*i*k+z*l*k-r*p-r*s
f7 = y*i*j+y*l*j+z*j*k+z*l^2-q*r-s^2

<!\textcolor{gray}{\# The polynomials corresponding to $c^{-1}b = bc$}!>
f8 = p*w+q*y-p*z+x*r
f9 = p*x+q*z-q*z+x*s
f10 = r*w+s*y-w*r+p*y
f11 = r*x+s*z-w*s+y*q

<!\textcolor{gray}{\# The polynomials corresponding to $ac^{-1}a^{-1} = cac$}!>
f12 = w^2*i+w*j*y+x*k*w+x*l*y-i*z*l-i*x*k+j*y*l+j*w*k
f13 = w*i*x+w*j*z+k*x^2+x*l*z-y*j^2-j*w*i+i*z*j+x*i^2
f14 = y*i*w+j*y^2+z*k*w+z*l*y-k*z*l-x*k^2+y*l^2+l*w*k
f15 = y*i*x+y*j*z+z*k*x+l*z^2-l*y*j-l*w*i+k*z*j+k*x*i

I = (f1,f2,f3,f4,f5,f6,f7,f8,f9,f10,f11,f12,f13,f14,f15)*R

I

	@Ideal (i*l-k*j-1, p*s-r*q-1, w*z-y*x-1, w*i^2+w*j*k+x*i*k+x*l*k-p^2-q*r, w*i*j+w*l*j+x*j*k+x*l^2-p*q-s*q, y*i^2+y*j*k+z*i*k+z*l*k-r*p-r*s, y*i*j+y*l*j+z*j*k+z*l^2-q*r-s^2, p*w+q*y-p*z+x*r, p*x+q*z-q*z+x*s, r*w+s*y-w*r+p*y, r*x+s*z-w*s+y*q, w^2*i+w*j*y+x*k*w+x*l*y-i*z*l-i*x*k+j*y*l+j*w*k, w*i*x+w*j*z+k*x^2+x*l*z-y*j^2-j*w*i+i*z*j+x*i^2, y*i*w+j*y^2+z*k*w+z*l*y-k*z*l-x*k^2+y*l^2+l*w*k, y*i*x+y*j*z+z*k*x+l*z^2-l*y*j-l*w*i+k*z*j+k*x*i) of Multivariate Polynomial Ring in i,j,k,l,p,q,r,s,w,x,y,z over Integer Ring@

(p^2+qr)^2+qr(p+s)(p+s)-1 in I

	@True@

q(p^2+qr)(p+s)+q(p+s)(qr+s^2) in I

	@True@
	
r(p+s)(p^2+qr)+r(qr+s^2)(p+s) in I

	@True@
	
qr(p+s)(p+s)+(qr+s^2)^2-1 in I

	@True@
\end{lstlisting}

\end{document}